\documentclass{amsart}
\usepackage{amsmath,amsthm,amssymb}
\usepackage{graphicx}
\usepackage{tikz-cd}
\usepackage{paralist}
\usepackage{environ}
\usepackage{xcolor}
\usepackage{hyperref}
\usepackage{centernot}

\usepackage{marvosym}
\usepackage{xcolor,cancel}

\usepackage{booktabs}
\usepackage{tabularx}
\usepackage{multirow,bigstrut}
\usepackage{bigstrut}
\usepackage{CJKutf8}

\makeatletter
\DeclareFontFamily{U}{tipa}{}
\DeclareFontShape{U}{tipa}{bx}{n}{<->tipabx10}{}
\newcommand{\arc@char}{{\usefont{U}{tipa}{bx}{n}\symbol{62}}}%

\newcommand{\arc}[1]{\mathpalette\arc@arc{#1}}

\newcommand{\arc@arc}[2]{%
  \sbox0{$\m@th#1#2$}%
  \vbox{
    \hbox{\resizebox{\wd0}{\height}{\arc@char}}
    \nointerlineskip
    \box0
  }%
}
\makeatother

\makeatletter
\newcommand{\doublewedge}{\big@doubleop{\wedge}}
\newcommand{\big@doubleop}[1]{%
  \DOTSB\mathop{\mathpalette\big@doubleop@aux{#1}}\slimits@
}

\newcommand\big@doubleop@aux[2]{%
  \sbox\z@{$\m@th#1#2$}%
  \makebox[1.35\wd\z@][s]{$\m@th#1#2\hss#2$}%
}

\newcommand{\abs}[1]{\left|#1\right|}     
\newcommand{\rb}{\mbox{rb}}
\newcommand{\re}{\mbox{re}}

\newcommand{\cl}{\mbox{cl}}
\newcommand{\dcl}{\mbox{cl}_{\Phi}}
\newcommand{\Int}{\mbox{int}}
\newcommand{\bdy}{\mbox{bdy}}
\newcommand{\er}{\mbox{He}}
\newcommand{\Hn}{\mbox{Hn}}
\newcommand{\Hb}{\mbox{Hb}}

\newcommand{\near}{\delta} 
\newcommand{\dnear}{\delta_{\Phi}} 
\newcommand{\dcap}{\mathop{\cap}\limits_{\Phi}} 

\newcommand{\sh}{\mbox{sh}}

\newcommand{\cyc}{\mbox{cyc}}

\usepackage{pstricks,pst-text,pst-grad,pst-node,pst-3dplot,pstricks-add,pst-poly,pst-coil,pst-bar,pst-all,pst-plot,pst-2dplot} 
\usepackage{pst-fun,pst-blur}
\usepackage{pst-all}
\usepackage{subfigure}
\renewcommand{\thesubfigure}{\thefigure.\arabic{subfigure}}
\makeatletter
\renewcommand{\p@subfigure}{}
\renewcommand{\@thesubfigure}{\thesubfigure:\hskip\subfiglabelskip}
\makeatother

\theoremstyle{plain}
\newtheorem{theorem}{Theorem}
\newtheorem{lemma}{Lemma}
\newtheorem{remark}{Remark}
\newtheorem{definition}{Definition}

\newtheorem{example}{Example}

\begin{document}

\title[Amiable Fixed Sets]{Amiable and Almost Amiable Fixed Sets.  Extension of the Brouwer Fixed Point Theorem}

\author[James F. Peters]{James F. Peters}
\address{
Computational Intelligence Laboratory,
University of Manitoba, WPG, MB, R3T 5V6, Canada and
Department of Mathematics, Faculty of Arts and Sciences, Ad\.{i}yaman University, 02040 Ad\.{i}yaman, Turkey}
\thanks{The research has been supported by the Natural Sciences \&
Engineering Research Council of Canada (NSERC) discovery grant 185986 
and Instituto Nazionale di Alta Matematica (INdAM) Francesco Severi, Gruppo Nazionale per le Strutture Algebriche, Geometriche e Loro Applicazioni grant 9 920160 000362, n.prot U 2016/000036 and Scientific and Technological Research Council of Turkey (T\"{U}B\.{I}TAK) Scientific Human
Resources Development (BIDEB) under grant no: 2221-1059B211301223.}


\date{}

\dedicatory{Dedicated to Sanjo Zlobec, honoring his 80th birthday}

%

\begin{abstract}
This paper introduces shape boundary regions in descriptive proximity forms of CW (Closure-finite Weak) spaces as a source of amiable fixed subsets as well as almost amiable fixed subsets of descriptive proximally continuous (dpc) maps. A dpc map is an extension of an Efremovi\v{c}-Smirnov proximally continuous (pc) map introduced during the early-1950s by V.A. Efremovi\v{c} and Yu. M. Smirnov. 
Amiable fixed sets and the Betti numbers of their free Abelian group representations are derived from dpc's relative to the description of the boundary region of the sets.  Almost amiable fixed sets are derived from dpc's by relaxing the matching description requirement for the descriptive closeness of the sets.  This relaxed form of amiable fixed sets works well for applications in which closeness of fixed sets is approximate rather than exact.
A number of examples of amiable fixed sets are given in terms of wide ribbons.  A bi-product of this work is a variation of the Jordan Curve Theorem and a Fixed Cell Complex Theorem, which is an extension of the Brouwer Fixed Point Theorem.
\end{abstract}

\maketitle
\tableofcontents

\section{Introduction}
This article introduces the boundary region of a cell complex in a Closure-finite Weak (CW) space as a source of amiable fixed sets (introduced in~\cite[\S 3, p. 9]{PetersVergili2020descriptiveFixedSets}).  Amiable fixed sets are a byproduct of descriptive proximally continuous maps that spawn fixed sets.  Briefly, the boundary region of a planar shape $\sh E$ (denoted by $\partial\sh E$) in a space $X$ is the set of all points in $X$ external to the closure of the shape.  In keeping with recent work on characterizing fixed points~\cite{SanjoZlobec2017fixedPoints}, this paper introduces a sufficient condition for a function to have a fixed set.  This leads to an extension of Brouwer's fixed point theorem and carries over to a class of continuous functions from a CW space to itself.

\begin{remark}
Let $\bdy(\sh E)$ be the contour (bounding edge) of a planar shape $\sh E$ in a space $X$. Contour $\bdy(\sh E)$ is a simple, closed curve. Let $\Int(\sh E)$ be the set of all points in the interior of $\sh E$.  
The closure of $\sh E = \bdy(\sh E)\cup\Int(\sh E)$ (denoted by $\cl(\sh E)$).  Let $\partial(\cl(\sh E))$ denote the boundary region external to $\cl(\sh E)$ and define
\begin{align*}
\partial(\cl(\sh E)) &= X\setminus (\bdy(\sh E)\cup\Int(\sh E)).\ \mbox{Then}\\
\bdy(\sh E) &= X\setminus (\Int(\sh E)\cup \partial(\cl(\sh E))).\\
\Int(\sh E) &= X\setminus (\bdy(\sh E)\cup \partial(\cl(\sh E))).
\end{align*}

That is, the boundary region $\partial\sh E$ of shape $\sh E$ is the set of all points not in $\cl(\sh E)$. The contour of a shape is commonly known as the \emph{boundary} of a simplicial complex and every such boundary is a cycle~\cite[p. 104]{Goblin2010CUPhomology}.
\end{remark}

A descriptive proximally continuous map~\cite[\S 1.7, p. 16]{Naimpally2013},~\cite[\S 5.1]{DiConcilio2018MCSdescriptiveProximities} is an extension of an Efremovi\v{c}-Smirnov proximally continuous map, first introduced by V.A. Efremovi\v{c}~\cite{Efremovic1952} and Yu. M. Smirnov~\cite{Smirnov1952,Smirnov1952a} in 1952.  In its simplest form, proximal continuity is defined on a \v{C}ech proximity space $(X,\near)$, in which $\near$ is a proximity relation that satisfies a number of axioms, especially $A\cap B\neq \emptyset$ implies $A\ \near\ B$ for subsets $A,B$ in $X$~\cite[\S 4, p. 20]{Naimpally70}.  A map $f:(X,\near)\to (X,\near)$ is \emph{proximally continuous}, provided, for each pair of subsets $A,B\subset X$, $A\ \near\ B$ implies $f(A)\ \near\ f(B)$.

Descriptive proximal continuity is an easy step beyond  proximal continuity. For a \v{C}ech proximity space $X$ containing a nonempty subset $A\in 2^X$ (collection of subsets of $X$), a probe function $\Phi:2^X\to \mathbb{R}^n$is defined by a feature vector $\Phi(A)$ that describes $A$, defined by\footnote{This form of probe function for the description of a set (suggested by A. Di Concilio) first appeared in~\cite{DiConcilio2018MCSdescriptiveProximities}.} $\Phi(A\in 2^X) = \left\{\Phi(a):a\in A\right\}$.  
Let $(X,{\dnear}_1), (Y,{\dnear}_2)$ be a pair of descriptive proximity spaces, equipped with proximity relations ${\dnear}_1,{\dnear}_2$.  For a pair of subsets $A,B\subset X$, $A\ {\dnear}_1 \ B$ reads $A$ is descriptively close to $B$. What we mean by \emph{descriptively close} depends on the definition of $\Phi(A\in 2^X)$. 

{\begin{example}
For a single-holed ring torus {\huge $\boldsymbol{\circledcirc}$} $E$ in space $X$, define $\Phi(e_1)$ = diameter of {\huge $\boldsymbol{\circledcirc}$} $E$ hole, $\Phi(e_2)$ = {\huge $\boldsymbol{\circledcirc}$} $E$ inner radius, and $\Phi(e_3)$ = {\huge $\boldsymbol{\circledcirc}$} $E$ outer radius.
Then $\Phi(\mbox{\huge $\boldsymbol{\circledcirc}$}$ $E)$ = $\left(\Phi(e_1),\Phi(e_2),\Phi(e_3)\right)$, feature vector that describes {\huge $\boldsymbol{\circledcirc}$} $E$.
\textcolor{blue}{\Squaresteel}
\end{example}

A map $f:(X,{\dnear}_1)\to (Y,{\dnear}_2)$ is a descriptive proximally continuous (dpc) map~\cite[\S 1.20.1, p. 48]{Peters2013springer}, provided there is at least one pair of subsets $A,B\subset X$ such that  $A\ {\dnear}_1\ B$ implies $f(A)\ {\dnear}_2\ f(B)$.
Two important bi-products of dpc maps are descriptive fixed sets and amiable fixed sets, provided the dpc maps satisfy certain conditions (see Def.~\ref{def:desProxCont}). 

For a dpc map $f$ on a CW space $X$ equipped with the proximity $\dnear$, a subset $E\subset X$ is fixed, provided the description of $f(E)$ matches the description of $E$.  Briefly, a nonvoid collection of cell complexes $X$ is a CW space~\cite[\S III, starting on page 124]{AlexandroffHopf1935Topologie},~\cite[pp. 315-317]{Whitehead1939homotopy}, ~\cite[\S 5, p. 223]{Whitehead1949BAMS-CWtopology}, provided $X$ is Hausdorff (every pair of distinct cells is contained in disjoint neighbourhoods~\cite[\S 5.1, p. 94]{Naimpally2013}) and the collection of cell complexes in $K$ satisfy the containment condition (the closure of each cell complex is in $K$ and each cell has only a finite number of faces~\cite{Switzer2002CWcomplex}) and intersection condition (the nonempty intersection of cell complexes is in $K$)
 

Results given here for dpc maps spring from the fundamental result for fixed points given by L.E.J. Brouwer~\cite{Brouwer1911fixedPoints}   

\begin{theorem}\label{thm:Brouwer}{\rm Brouwer Fixed Point Theorem~\cite[\S 4.7, p. 194]{Spanier1966AlgTopology}}$\mbox{}$\\
Every continuous map from $\mathbb{R}^n$ to itself has a fixed point.
\end{theorem}

\begin{figure}[!ht]
\centering
\subfigure[Shape $\sh E = \blacktriangle E_1\cup \left\{\blacktriangle E_2\cup \blacktriangle E_3\right\}$ (3 filled triangles with a common vertex $v_0$)\ \&\ shape boundary region $\partial (\sh E)$]
 {\label{fig:shEBoundary}
\begin{pspicture}
(-1.5,-0.8)(4.0,3.0)
\psframe*[linewidth=0.75pt,linearc=0.25,cornersize=absolute,linestyle=solid,linecolor=orange!20](-1.25,-0.25)(3.25,3)
\psframe[linewidth=0.75pt,linearc=0.25,cornersize=absolute,linestyle=solid](-1.25,-0.25)(3.25,3)
\psline*[linecolor = green!50]%
(0.0,1)(1.85,1.85)(2.95,1.25)(0.0,1)
(-0.3,2)(1,2.5)(1.85,1.85)(0.0,1)
\psline[linecolor = black]%
(0.0,1)(1.85,1.85)(0.0,1)(2.95,1.25)(0.0,1)
(-0.3,2)(1,2.5)(1.85,1.85)(0.0,1)
\psline[linecolor = black]%
(0.0,1)(1,2.5)
\psline[linecolor = black]%
(1.85,1.85)(2.95,1.25)
\psdots[dotstyle=o, dotsize=1.3pt 2.25,linecolor = black, fillcolor = yellow]%
(0.0,1)(1.85,1.85)(2.95,1.25)(0.0,1)
(-0.3,2)(1,2.5)(1.85,1.85)
\psline[linestyle=solid,linecolor=blue,border=1pt]{*->}(2.5,3.25)(-0.5,2.50)
\psline[linestyle=solid,linecolor=blue,border=1pt]{*->}(2.5,3.25)(2.3,2.15)
\psline[linestyle=solid,linecolor=blue,border=1pt]{*->}(2.5,-0.5)(1.3,1.12)
\rput(2.5,3.45){\footnotesize $\boldsymbol{\partial (\sh E)}$}
\rput(-1.0,2.75){\footnotesize $\boldsymbol{K}$}
\rput(0.3,2.0){\footnotesize $\boldsymbol{\blacktriangle E_1}$}
\rput(1.0,2.0){\footnotesize $\boldsymbol{\blacktriangle E_2}$}
\rput(1.8,1.5){\footnotesize $\boldsymbol{\blacktriangle E_3}$}
\rput(-0.2,0.8){\footnotesize $\boldsymbol{v_0}$}
\rput(2.5,-0.7){\footnotesize $\boldsymbol{sh E}$}
\end{pspicture}}\hfil
\subfigure[Shape $\sh E' = \blacktriangle E'_1\cup \left\{\blacktriangle E'_2\cup \blacktriangle E'_3\right\}$ (3 filled triangles with a common vertex $v'_0$)\ \&\ shape boundary region $\partial (\sh E')$]
 {\label{fig:shEBoundary2}
\begin{pspicture}
(-1.5,-0.8)(4.0,3.5)
\psframe*[linewidth=0.75pt,linearc=0.25,cornersize=absolute,linestyle=solid,linecolor=orange!20](-1.25,-0.25)(3.25,3)
\psframe[linewidth=0.75pt,linearc=0.25,cornersize=absolute,linestyle=solid](-1.25,-0.25)(3.25,3)
\psline*[linecolor = gray!20]%
(0.0,1)(1.85,1.85)(2.95,1.25)(0.0,1)
(-0.8,2)(1,2.5)(1.85,1.85)(0.0,1)
\psline[linecolor = black]%
(0.0,1)(1.85,1.85)(0.0,1)(2.95,1.25)(0.0,1)
(-0.8,2)(1,2.5)(1.85,1.85)(0.0,1)
\psline[linecolor = black]%
(0.0,1)(1,2.5)
\psline[linecolor = black]%
(1.85,1.85)(2.95,1.25)
\psdots[dotstyle=o, dotsize=1.3pt 2.25,linecolor = black, fillcolor = yellow]%
(0.0,1)(1.85,1.85)(2.95,1.25)(0.0,1)
(-0.8,2)(1,2.5)(1.85,1.85)
\psline[linestyle=solid,linecolor=blue,border=1pt]{*->}(2.5,3.25)(-0.5,2.50)
\psline[linestyle=solid,linecolor=blue,border=1pt]{*->}(2.5,3.25)(2.3,2.15)
\psline[linestyle=solid,linecolor=blue,border=1pt]{*->}(2.5,-0.5)(1.3,1.12)
\rput(2.5,3.45){\footnotesize $\boldsymbol{\partial (\sh E')}$}
\rput(-1.0,2.75){\footnotesize $\boldsymbol{K'}$}
\rput(0.3,2.0){\footnotesize $\boldsymbol{\blacktriangle E'_1}$}
\rput(1.0,2.0){\footnotesize $\boldsymbol{\blacktriangle E'_2}$}
\rput(1.8,1.5){\footnotesize $\boldsymbol{\blacktriangle E'_3}$}
\rput(-0.2,0.8){\footnotesize $\boldsymbol{v'_0}$}
\rput(2.5,-0.7){\footnotesize $\boldsymbol{sh E'}$}
\end{pspicture}}\hfil
\caption[]{Sample shapes and shape boundary regions}
\label{fig:boundaryRegions2}
\end{figure}
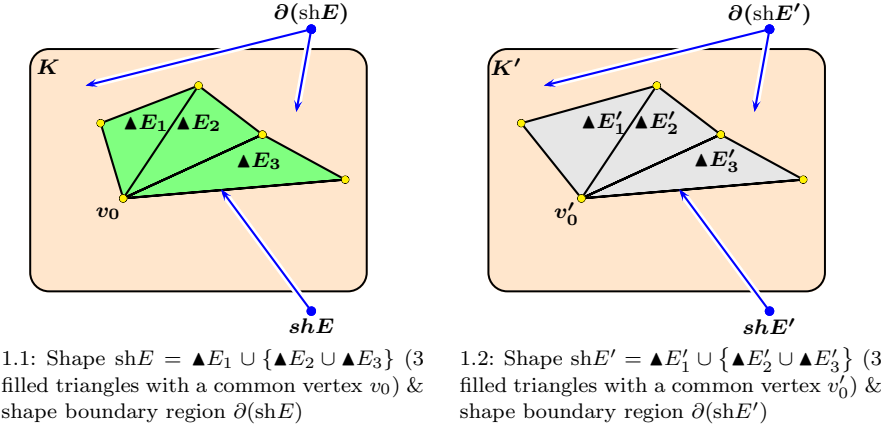

\section{Preliminaries}
The simplest form of proximity relation (denoted by $\delta$) on a nonempty set was introduced by E. \v{C}ech~\cite{Cech1966}.  A nonempty set $X$ equipped with the relation $\near$ is a \v{C}ech proximity space (denoted by $(X,\near$)), provided the following axioms are satisfied.\\
\vspace{3mm}

\noindent {\bf \v{C}ech Axioms}

\begin{description}
\item[({\bf P}.0)] All nonempty subsets in $X$ are far from the empty set, {\em i.e.}, $A\ \not{\near}\ \emptyset$ for all $A\subseteq X$.
\item[({\bf P}.1)] $A\ \near\ B \Rightarrow B\ \near\ A$.
\item[({\bf P}.2)] $A\ \cap\ B\neq \emptyset \Rightarrow A\ \near\ B$.
\item[({\bf P}.3)] $A\ \near\ \left(B\cup C\right) \Rightarrow A\ \near\ B$ or $A\ \near\ C$.
\end{description} 

\noindent A filled triangle is a triangle with a nonvoid interior.  Sample filled triangles are shown in Figure~\ref{fig:boundaryRegions2}.

\begin{example}{\bf Proximal Filled Triangles}.\\
Let $(K,\near)$ be a \v{C}ech proximity space, represented in Figure~\ref{fig:boundaryRegions2}.  The space $K$ contains a shape $\sh E$ constructed from three filled triangles $\blacktriangle E_1,\blacktriangle E_2,\blacktriangle E_3$.  We have
\begin{align*}
\blacktriangle E_1\ &\cap\ \left\{\blacktriangle E_2\cup \blacktriangle E_3\right\} \neq \emptyset.\ \mbox{Hence,}\\
\blacktriangle E_1\ &\near\ \left\{\blacktriangle E_2\cup \blacktriangle E_3\right\}\ \mbox{(from Axiom {\bf P.2})}.
\end{align*}
Many other examples of \v{C}ech proximities can be found in shape $\sh E$.
\textcolor{blue}{\Squaresteel}
\end{example}

 Given that a nonempty set $E$ has $k \geq 1$ features such as Fermi energy $E_{Fe}$, cardinality $E_{card}$, a description $\Phi(E)$ of $E$ is a feature vector, {\em i.e.}, $\Phi(E) = \left(E_{Fe},E_{card}\right)$. Nonempty sets $A,B$ with overlapping descriptions are descriptively proximal (denoted by $A\ \dnear\ B$). 
  
\begin{definition}\label{def:dcap}{\rm \bf Descriptive Intersection}.\\	
There are two cases to consider.
\begin{compactenum}[1$^o$]	
\item Let $A,B\in 2^X$, nonempty subsets in a collection of subsets in a space $X$. The descriptive intersection\cite{DiConcilio2018MCSdescriptiveProximities} of nonempty subsets in $A\cup B$ (denoted by $A\ \dcap\ B$) is defined by
\[
A\ \dcap\ B = \overbrace{\left\{x\in A\cup B: \Phi(x) \in \Phi(A)\ \cap\ \Phi(B)\right\}.}^{\mbox{\textcolor{blue}{\bf {\em i.e.}, $\boldsymbol{\mbox{Descriptions}\ \Phi(A)\ \&\ \Phi(B)\ \mbox{overlap}}$ and possibly $A\cap B \neq \emptyset$.}}}
\]
\item Let $A\in 2^X, B\in 2^{X'}$, nonempty subsets in collections of subsets in distinct spaces $X, X'$.  In that case, $A\cap B = \emptyset$ and $A\ \dcap\ B$ is defined by
\[
A\ \dcap\ B = \overbrace{\left\{x\in A\ \mbox{or}\ x\in B: \Phi(x) \in \Phi(A)\ \cap\ \Phi(B)\right\}.}^{\mbox{\textcolor{blue}{\bf {\em i.e.}, $\boldsymbol{\mbox{Descriptions}\ \Phi(A)\ \&\ \Phi(B)\ \mbox{overlap}}$ but always $A\cap B = \emptyset$.}}}
\]
\end{compactenum}
In the context of dpc maps $f:(X,\dnear)\to (X,\dnear)$ or $f:(X,\dnear)\to (X',\dnear)$, the presence of a nonempty $f(A)\ \dcap\ A$ yields amiable fixed sets (cf. Def.~\ref{def:desProxCont}). \textcolor{blue}{\Squaresteel}
\end{definition}

\begin{remark}
From Def.~\ref{def:dcap}, amiable fixed sets can be subsets in dpc maps from a proximity space to itself or from dpc maps from one proximity space $X$ to a different proximity space $X'$.  The model for descriptive intersection for the second form of dpc is the more natural one in the physical world, where it is common that to find objects in different spaces ({\em e.g.}, distinct spaces separated by time and place) that have matching descriptions.
\textcolor{blue}{\Squaresteel}
\end{remark}

Let $2^X$ denote the collection of all subsets in a nonvoid set $X$. A nonempty set $X$ equipped with the relation $\dnear$ with non-void subsets $A,B,C\in 2^X$ is a descriptive proximity space, provided the following descriptive forms of the \v{C}ech axioms are satisfied.\\
\vspace{3mm}

\noindent {\bf Descriptive \v{C}ech Axioms}

\begin{description}
\item[({\bf dP}.0)] All nonempty subsets in $2^X$ are descriptively far from the empty set, {\em i.e.}, $A \not{\dnear}\ \emptyset$ for all $A\in 2^X$.
\item[({\bf dP}.1)] $A\ \dnear\ B \Rightarrow B\ \dnear\ A$.
\item[({\bf dP}.2)] $A\ \dcap\ B\neq \emptyset \Rightarrow A\ \dnear\ B$.
\item[({\bf dP}.3)] $A\ \dnear\ \left(B\cup C\right) \Rightarrow A\ \dnear\ B$ or $A\ \dnear\ C$.
\end{description} 

\noindent The converse of Axiom ({\bf dp}.2) also holds.
\begin{lemma}\label{lemma:dP2converse}{\rm \cite{Peters2019vortexNerves}}
Let $X$ be equipped with the relation $\dnear$, $A,B\in 2^X$.   Then $A\ \dnear\ B$ implies $A\ \dcap\ B\neq \emptyset$.
\end{lemma}
\begin{proof}
Let $A,B\in 2^X$. By definition, $A\ \dnear\ B$ implies that there is at least one member $x\in A$ and $y\in B$ so that $\Phi(x) = \Phi(y)$, {\em i.e.}, $x$ and $y$ have the same description.  Then $x,y\in A\ \dcap\ B$.
Hence, $A\ \dcap\ B\neq \emptyset$, which is the converse of ({\bf dp}.2).    
\end{proof}

The Betti number $\beta_0$~\cite{Zomorodian2001BettiNumbers} is a count of the number of cells in a cell complex.  Here,  $\beta_0$ is a count of the number of filled triangles in a shape.  

\begin{example}\label{ex:desProxShapes}{\bf Descriptively Proximal Shapes}.\\
Let $(K,\dnear),(K',\dnear)$ be a pair of descriptive proximity spaces, represented in Fig.~\ref{fig:boundaryRegions2} defined in terms of the same descriptive proximity $\dnear$ with
\[
\Phi(E) = \beta_0(E)\ \mbox{for $E\subset K$}.
\]
For a pair of shapes in two different spaces $K, K'$, we have $\sh E\ \dcap\ \sh E'\neq \emptyset$, since
\[
\Phi(\sh E) = \beta_0(\sh E) = \Phi(\sh E') = \beta_0(\sh E') = 3.
\]
That is, both shapes have the same description, since both shapes are constructed with 3 filled triangles. 
Hence, from Axiom {\bf dP.2}, $\sh E\ \dnear\ \sh E'$.
\textcolor{blue}{\Squaresteel}
\end{example}

\begin{table}[!ht]\scriptsize 
\caption{Minimal Planar Cell Complexes}
\label{tab:skeleton}
\begin{tabular}{|c|c|c|c|}
    \hline
    Minimal Complex & Type $K_i, i = 0, 1, 2$  & Planar Geometry & Interior\\
    \hline
    \hline
\begin{pspicture}
(0,0)(1,1)
\psdots[dotstyle=o, linewidth=1.2pt,linecolor = black, fillcolor = black]%
(0.5,0.5)	
\end{pspicture}		
& 
$K_0$
&
Vertex
&
nonempty
\\
    \hline
    \hline
\begin{pspicture}
(0,0)(1,1)
\psline[showpoints=true,linestyle=solid,linecolor = black]%
(0.25,0.25)(0.75,0.75)
\psdots[dotstyle=o, dotsize=1.3pt 2.25,linecolor = black, fillcolor = black]%
(0.25,0.25)(0.75,0.75)
\end{pspicture}& 
$K_1$
&
Edge
&
nonempty
\\
\hline
\hline
\begin{pspicture}
(0,0)(1,1)
\psline*[linecolor = green!50]%
(0.0,0.15)(0.85,0.85)(0.95,0.25)
\psline[linecolor = black]%
(0.0,0.15)(0.85,0.85)(0.95,0.25)(0.0,0.15)
\psdots[dotstyle=o, dotsize=1.3pt 2.25,linecolor = black, fillcolor = black]%
(0.0,0.15)(0.85,0.85)(0.95,0.25)
\end{pspicture}& 
$K_2$
&
Filled triangle
&
nonempty
\\
\hline
\end{tabular}
\end{table}

\section{CW Spaces}
In this work, a CW space is a collection of cells in a finite, bounded region of Euclidean space.   A \emph{cell} in the Euclidean plane is either a 0-cell (vertex $K_0$) or 1-cell (edge $K_1$) or 2-cell (filled triangle $K_2$).  A {\bf cell complex} is a collection cells attached to each other by 0- or 1-cells or by having one or more common cells. The planar geometry of minimal cell complexes with nonempty interiors are displayed in Table~\ref{tab:skeleton}. Collections of filled cycles ({\em e.g.}, Fig.~\ref{fig:threeCycles}) and ribbons ({\em e.g.}, Fig.~\ref{fig:boundaryRegions}) are examples of cell complexes.

\begin{figure}[!ht]
\centering
\begin{pspicture}
(-2.5,-1.8)(4,2.5)
\psframe[linewidth=0.75pt,linearc=0.25,cornersize=absolute,linestyle=solid](-0.8,-1.5)(4.8,1.8)
\psline[linestyle=solid]%
(0,0)(1,1)(3,1)(4,0)(3,-1)(1,-1)(0,0)
\psline*[linestyle=solid,linecolor=green!30]%
(0,0)(1,1)(3,1)(3,-1)(1,-1)(0,0)
\psline[linestyle=solid,linecolor=black]%
(1,1)(1.8,0.2)(3,1)(3,-1)(1.8,0.2)(1,-1)
\psline[linewidth=0.95pt,linecolor=red,arrowsize=0pt 5]{<->}(0,0)(1,1)
\psline[linecolor=blue,arrowsize=0pt 5]{<->}(1,1)(3,1)
\psline[linewidth=0.95pt,linecolor=black,arrowsize=0pt 5]{<->}(3,1)(3,-1)
\psline[linewidth=0.95pt,linecolor=black,arrowsize=0pt 5]{<->}(3,1)(4,0)
\psline[linewidth=0.95pt,linecolor=black,arrowsize=0pt 5]{<->}(4,0)(3,-1)
\psline[linecolor=blue,arrowsize=0pt 5]{<->}(3,-1)(1,-1)
\psline[linewidth=0.95pt,linecolor=red,arrowsize=0pt 5]{<->}(1,-1)(0,0)
\psline[linewidth=0.95pt,linecolor=red,arrowsize=0pt 5]{<->}(1,1)(1.8,0.2)
\psline[linewidth=0.95pt,linecolor=red,arrowsize=0pt 5]{<->}(1.8,0.2)(1,-1)
\psline[linecolor=blue,arrowsize=0pt 5]{<->}(3,1)(1.8,0.2)
\psline[linecolor=blue,arrowsize=0pt 5]{<->}(1.8,0.2)(3,-1)
\psdots[dotstyle=o, linewidth=1.2pt,linecolor = black, fillcolor = yellow]%
(0,0)(1,-1)(1,1)(1.8,0.2)(3,1)(3,-1)(1.8,0.2)(4,0)
\psline[linestyle=solid,linecolor=blue,border=1pt]{*->}(0,2.1)(0.5,0.5) 
\rput(0,2.3){$\boldsymbol{cyc E_{min-filled}}$}
\psline[linestyle=solid,linecolor=blue,border=1pt]{*->}(4,2.1)(3.5,0.5) 
\rput(4,2.3){$\boldsymbol{cyc E_{non-filled}}$}
\rput(1,1.25){$\boldsymbol{v_1}$}
\rput(3.0,1.25){$\boldsymbol{v_2}$}
\rput(1,-1.25){$\boldsymbol{v_3}$}
\rput(3.0,-1.25){$\boldsymbol{v_4}$}
\rput(1.8,0.45){$\boldsymbol{v_5}$}
\rput(4.3,0){$\boldsymbol{v_8}$}
\rput(-0.3,-0.0){$\boldsymbol{v_0}$}
\rput(-0.5,1.5){\footnotesize $\boldsymbol{K}$}
\end{pspicture}
\caption[]{Shape $\sh E$ in CW space $K$ is a collection of cycles with vertexes $v_0, v_1, v_3, v_5$ residing in minimal-length filled cycle $\boldsymbol{cyc E_{min-filled}}$\ \&\ vertexes $v_2, v_4, v_8$ in non-filled cycle $\boldsymbol{cyc E_{non-filled}}$.}
\label{fig:1-cyclePaths}
\end{figure}
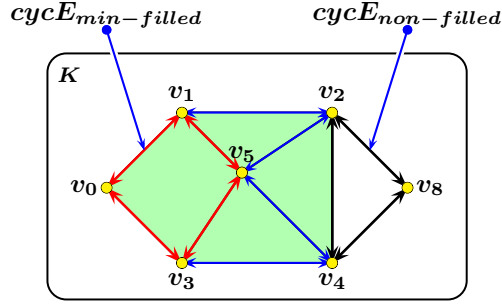

\begin{example}
A collection of cell complexes defined by overlapping, filled and non-filled 1-cycles that define a shape $\sh E$ in a CW space $K$ is shown in Fig.~\ref{fig:1-cyclePaths}.  The white areas represent the boundary region of $\sh E$. Let $v_1, v_2, v_3, v_4, v_5, v_8$ be path-connected vertexes on bi-directional edges in a number of different cycles such as
\begin{align*}
\cyc E_{min-filled} &= \overbrace{v_0\to v_1\to v_5\to
					v_3\to v_0\ \mbox{(traversal from $v_0$ to $v_0$)}.}^{\mbox{\textcolor{blue}{\bf sequence of maps in traversal on $\boldsymbol{E_{min-filled}}$}}}\\
E_{non-filled} &= \overbrace{v_2\to v_8\to
          v_4\to v_2}^{\mbox{\textcolor{blue}{\bf sequence of maps in traversal on $\boldsymbol{\cyc E_{non-filled}}$}}}         
\end{align*}

The filled cycles in Fig.~\ref{fig:1-cyclePaths} have nonvoid interiors represented by \textcolor{green!80}{\Large \Squaresteel}. Filled cycle $E_{min-filled}$ has the shortest path that starts and ends with $v_0$.  The path defined by $E_{min-filled}$ is simple (it has no loops) and closed (each complete traversal that starts with a particular vertex ends with the same vertex).   Cycle $\cyc E_{non-filled}$ is non-filled cycle, since its interior is part of the white boundary region.
\qquad \textcolor{blue}{\Squaresteel}
\end{example}

A nonvoid collection of cell complexes $K$ has a \emph{Closure finite Weak} (CW) topology, provided $K$ is Hausdorff (every pair of distinct cells is contained in disjoint neighbourhoods~\cite[\S 5.1, p. 94]{Naimpally2013}) and the collection of cell complexes in $K$ satisfy the Alexandroff-Hopf-Whitehead~\cite[\S III, starting on page 124]{AlexandroffHopf1935Topologie},~\cite[pp. 315-317]{Whitehead1939homotopy}, ~\cite[\S 5, p. 223]{Whitehead1949BAMS-CWtopology} conditions, namely, containment (the closure of each cell complex is in $K$) and intersection (the nonempty intersection of cell complexes is in $K$). 

A number of important results in this paper spring from descriptive proximally continuous maps.  

\begin{definition}\label{def:desProxCont}{\rm~\cite[\S 3, p. 8]{PetersVergili2020descriptiveFixedSets} {\bf Descriptive Fixed Sets and Amiable Fixed Sets.}}\\
	Let $(X,\delta_{\Phi})$ be a descriptive  \v{C}ech proximity space and $f: (X, \delta_{\Phi}) \to (X, \delta_{\Phi})$ a descriptive proximally continuous map (our dpc map) . 
		\begin{compactenum}[(i)]
\item $A$  is  a descriptive fixed subset of $f$, provided  $\Phi(f(A))=\Phi(A)$. 
\item $A$ and $f(A)$ are amiable fixed sets, provided $f(A) \ \dcap \ A \neq \emptyset$. 
\end{compactenum}
\end{definition}

Let $K$ be a CW space.  Briefly, the boundary $\partial A$ of a cell complex $A$ in $K$ is the set of all points not in $A$.

\begin{example}
Let the CW space $K$ be represented by Fig.~\ref{fig:shEBoundary}.  
Let $f: (K, \delta_{\Phi}) \to (K, \delta_{\Phi})$ be a descriptive proximally continuous map defined by 
\[
f(A) = K\setminus \partial A,\ \mbox{for cell complex}\ A\in K.
\]
For shape $\sh E$ in Fig.~\ref{fig:shEBoundary}, $f(\sh E)= K\setminus \partial \sh E = \sh E.$  From Example~\ref{ex:desProxShapes}, $\Phi(f(\sh E))=\Phi(\sh E)$.  Hence, $\sh E$  is  a descriptive fixed subset of $f$.\\
Also, from Example~\ref{ex:desProxShapes}, $f(\sh E)$ and shape $\sh E$ are amiable, since $f(\sh E) \ \dcap \ \sh E \neq \emptyset$.
\textcolor{blue}{\Squaresteel}
\end{example} 

\begin{figure}[!ht]
\centering
\subfigure[Cycle with boundary $\bdy E$\ \&\ interior $\Int E$\ \&\ cyclic group generator vertex $v_0$]
 {\label{fig:cycE}
\begin{pspicture}
(-1.5,-0.5)(4.0,3.5)
\psframe[linewidth=0.75pt,linearc=0.25,cornersize=absolute,linestyle=solid](-1.25,-0.45)(3.25,3)
\psframe[linewidth=0.75pt,linearc=0.25,cornersize=absolute,linestyle=solid](-1.25,-0.45)(3.25,3)
\psline*[linestyle=solid,linecolor=green!30]%
(0,0)(1,0.5)(2.0,0.0)(3.0,0.5)(3.0,1.5)(2.0,2.0)(1,1.5)(0,2)
(-1,1.5)(-1,0.5)(0,0)
\psline[linestyle=solid,linecolor=black]%
(0,0)(1,0.5)(2.0,0.0)(3.0,0.5)(3.0,1.5)(2.0,2.0)(1,1.5)(0,2)
(-1,1.5)(-1,0.5)(0,0)
\psdots[dotstyle=o,dotsize=2.2pt,linewidth=1.2pt,linecolor=black,fillcolor=yellow!80]%
(0,0)(1,0.5)(2.0,0.0)(3.0,0.5)(3.0,1.5)(2.0,2.0)(1,1.5)(0,2)
(-1,1.5)(-1,0.5)(0,0)
\psline[linestyle=solid,linecolor=blue,border=1pt]{*->}(2.5,3.25)(2.2,1.90) 
\psline[linestyle=solid,linecolor=blue,border=1pt]{*->}(1.0,3.15)(1.8,0.90) 
\rput(2.5,3.45){\footnotesize $\boldsymbol{\bdy(\cyc E)}$}
\rput(1.0,3.35){\footnotesize $\boldsymbol{\Int(\cyc E)}$}
\rput(-1.0,2.75){\footnotesize $\boldsymbol{K}$}
\rput(0.0,2.20){\footnotesize $\boldsymbol{\cyc E}$}
\rput(1.0,1.70){\footnotesize $\boldsymbol{v_0}$}
\rput(0.0,-0.2){\footnotesize $\boldsymbol{v}$}
\end{pspicture}}\hfil
\subfigure[Intersecting cycles $\cyc A, \cyc B$ with $\cyc A\cap \cyc B = v$\ \&\ overlapping interiors, {\em i.e}, $\Int(\cyc B)\subset \Int(\cyc A)$\ \&\ free finitely-generated Abelian group with generator vertices $v_0,v_{o'}$]
 {\label{fig:intersectingCycles}
\begin{pspicture}
(-1.5,-0.5)(4.0,3.0)
\psframe[linewidth=0.75pt,linearc=0.25,cornersize=absolute,linestyle=solid](-1.25,-0.45)(3.25,3)
\psframe[linewidth=0.75pt,linearc=0.25,cornersize=absolute,linestyle=solid](-1.25,-0.45)(3.25,3)
\psline*[linestyle=solid,linecolor=green!30]%
(0,0)(1,0.5)(2.0,0.0)(3.0,0.5)(3.0,1.5)(2.0,2.0)(1,1.5)(0,2)
(-1,1.5)(-1,0.5)(0,0)
\psline[linestyle=solid,linecolor=black]%
(0,0)(1,0.5)(2.0,0.0)(3.0,0.5)(3.0,1.5)(2.0,2.0)(1,1.5)(0,2)
(-1,1.5)(-1,0.5)(0,0)
\psdots[dotstyle=o,dotsize=2.2pt,linewidth=1.2pt,linecolor=black,fillcolor=yellow!80]%
(0,0)(1,0.5)(2.0,0.0)(3.0,0.5)(3.0,1.5)(2.0,2.0)(1,1.5)(0,2)
(-1,1.5)(-1,0.5)(0,0)
\psline*[linestyle=solid,linecolor=green!30]%
(0,0.0)(1,0.75)(2.0,0.25)(2.5,0.5)(2.5,0.75)(2.0,1.35)(1,1.25)(0,1.5)
(-.55,1.25)(-.55,0.75)(0,0.0)
\psline[linestyle=solid,linecolor=black]%
(0,0.0)(1,0.75)(2.0,0.25)(2.5,0.5)(2.5,0.75)(2.0,1.35)(1,1.25)(0,1.5)
(-.55,1.25)(-.55,0.75)(0,0.0)
\psdots[dotstyle=o,dotsize=2.2pt,linewidth=1.2pt,linecolor=black,fillcolor=yellow!80]%
(0,0.0)(1,0.75)(2.0,0.25)(2.5,0.5)(2.5,0.75)(2.0,1.35)(1,1.25)(0,1.5)
(-.55,1.25)(-.55,0.75)(0,0.0)
\psdots[dotstyle=o,dotsize=2.2pt,linewidth=1.2pt,linecolor=black,fillcolor=red]%
(0,0.0)
\psline[linestyle=solid,linecolor=blue,border=1pt]{*->}(2.5,3.25)(2.2,1.90) 
\psline[linestyle=solid,linecolor=blue,border=1pt]{*->}(-0.5,3.25)(-0.2,1.45) 
\psline[linestyle=solid,linecolor=blue,border=1pt]{*->}(1.0,3.15)(1.8,0.90) 
\psline[linestyle=solid,linecolor=blue,border=1pt]{*->}(1.0,3.15)(1.8,1.60) 
\psline[linestyle=solid,linecolor=blue,border=1pt]{*->}(0.3,3.55)(0.35,0.70) 
\rput(2.5,3.45){\footnotesize $\boldsymbol{\bdy(\cyc A)}$}
\rput(-0.5,3.45){\footnotesize $\boldsymbol{\bdy(\cyc B)}$}
\rput(1.0,3.35){\footnotesize $\boldsymbol{\Int(\cyc A)}$}
\rput(0.5,3.75){\footnotesize $\boldsymbol{\Int(\cyc B)}$}
\rput(-1.0,2.75){\footnotesize $\boldsymbol{K}$}
\rput(1.0,1.70){\footnotesize $\boldsymbol{v_0}$}
\rput(1.0,0.90){\footnotesize $\boldsymbol{v'_{0}}$}
\rput(0.0,-0.2){\footnotesize $\boldsymbol{v}$}
\rput(2.8,1.85){\footnotesize $\boldsymbol{cyc A}$}
\rput(2.5,1.25){\footnotesize $\boldsymbol{cyc B}$}
\end{pspicture}}\hfil
\caption[]{Sample filled cycles $\cyc E$ and $\cyc A, \cyc B$}
\label{fig:threeCycles}
\end{figure}

\section{1-Cycles, Cyclic group representations and Betti numbers}
This section revisits filled 1-cycles (briefly, cycles), which are fundamental structures in a CW space $K$.  

\begin{definition}\label{def:filledCycle} {\rm \bf Filled Cycle}.\\
A filled cycle $\cyc E$ in a CW space $K$ is defined by
\begin{align*}
v,v_0 &\in \bdy E\ \mbox{(contains vertexes $v$\ \&\ generator vertex $v_0$)}.\\
\bdy(\cyc E) &= \overbrace{\left\{v\in K: \exists v_0, v = kv_0\right\}.}^{\mbox{\textcolor{blue}{\bf \mbox{each} $\boldsymbol{v}$\ \mbox{can be reached by}\ $\boldsymbol{k}$\ \mbox{moves from $v_0$}}}}\\
       &\overbrace{\mbox{collection of path-connected vertexes with no end vertex.}}^{\mbox{\textcolor{blue}{\bf defines a simple closed curve}}}\\
\Int(\cyc E) &= \ \mbox{nonvoid interior of $\boldsymbol{\cyc E}$, such that}\\
\cyc E &= \bdy(\cyc E) \cup \Int(\cyc E).
\end{align*}
In other words, a filled cycle is a bounded region with a simple, closed curve defined by a collection of path-connected vertexes with no end vertex and a nonvoid interior within its boundary.
\textcolor{blue}{\Squaresteel}
\end{definition}


\begin{example}\label{ex:filledCycle}
A filled cycle $\cyc E$ with boundary $\bdy(\cyc E)$\ and interior $\Int(\cyc E)$\ containing a cyclic group generator vertex $v_0$  is shown in Fig.~\ref{fig:cycE}.  For each vertex $v\in \bdy(\cyc E)$, there is a minimal sequence of $k$ moves in a traversal of the edges between $v_0$ and $v$.  For this reason, we write
\[
v = \overbrace{kv_0.}^{\mbox{\textcolor{blue}{\bf $k$ moves to reach $v$ from $v_0$}}}\mbox{\textcolor{blue}{\Squaresteel}}
\]
\end{example}

\begin{theorem}\label{thm:filledCycle}
The closure of a planar shape is a filled 1-cycle in a CW space.
\end{theorem}
\begin{proof}
Let $\cl(\sh E)$ be the closure of a shape $\sh E$, which is a cell complex in CW space.  From Lemma 4.15~\cite{Peters2020CGTPhysics}, we have
\[
(\alpha)\qquad cl(\sh E) = \bdy(\sh E) \cup \Int(\sh E).
\]
The boundary of a planar shape is a simple, closed curve that can be decomposed into a collection of path-connected vertexes with no end vertex.  Consequently, a shape boundary is the boundary of a filled 1-cycle.
Hence, $\cl(\sh E)$ is a filled 1-cycle (from ($\alpha$) and Def.~\ref{def:filledCycle}).
\end{proof}

\begin{remark}
An application of filled cycles viewed as planar models of wide ribbons can be found in tracking the determination of final loci of fluid circulation~\cite[\S 5.5]{Zhang-Li2020fluidFlow}.  A slice of a moving fluid such as water in a tube or in an ocean wave approaching the slope of a shoreline is a collection of nested, usually non-concentric filled cycles.
By considering an expander graph approach to nested filled cycles such as Vigolo wide ribbon expanders~\cite{Vigolo2018HawaiianEarrings},
we obtain a formal method useful in the study of fluid flow as well as the interaction of a moving fluid along its boundaries.
\textcolor{blue}{\Squaresteel}
\end{remark}

Recall that a cyclic group is a nonempty set $G$ with a binary operation $+$ and element $g\in G$ (called a generator, denoted by $\langle g\rangle$) so that every other element $g'\in G$ can be written as a multiple of $g$, {\em i.e.}, 
\[
g' = \overbrace{g + \cdots + g = kg.}^{\mbox{\textcolor{blue}{\bf $g'$ is a linear combination of generator $g$}}}
\]
We write $G(\langle g\rangle,+)$ to denote a cyclic group with a single generator.

The vertex $v_0$ in Example~\ref{ex:filledCycle} serves as a generator of a cyclic group.  A filled cycle 
Let $0v$ represent zero moves from any vertex $v\in \bdy E$.  Moving $k$ moves backward from vertex $v$ is denoted by $-kv$.  This gives us $kv - kv = (k-k)v = 0v$. For $k = 1$, $v - v = (1-1)v = 0v$ ({\em i.e.}, $-v$ is the additive (move) inverse of $v$.  Let $v,v'\in \bdy E$ and let $v + v'$ represent a move from vertex $v$ to vertex $v'$.  Then we have
\begin{align*}
k,k' &> 0, k\neq k'\\
v + v' &= kv_0 + k'v_0\\
       &= (k + k')v_0\\
			 &= (k' + k)v_0\\
			 &= k'v_0 + kv_0\\
			 &= v' + v.
\end{align*} 
Hence, the move operation + is Abelian and $G(\langle v_0\rangle,+)$ is a cyclic group that is Abelian.  This gives us the result in Lemma~\ref{lem:cyclicGroup}.

\begin{lemma}\label{lem:cyclicGroup}
The boundary of a filled cycle has an Abelian cyclic group representation.
\end{lemma}

\begin{definition}\label{def:freeAbelianGroup} ({\bf Free Finitely-Generated Abelian Group and its Betti Number})
A \emph{free finitely-generated (fg) Abelian} group $G$ with binary operation $+$ is a group containing $n$ cyclic subgroups (each with its own generator)~\cite[\S A.9, pp. 217-218]{Goblin2010CUPhomology}, denoted by $G\left(\langle g_1,\dots,g_n\rangle,+\right)$. An Abelian group $G$ is \emph{free}\cite[\S 1.2, p. 2]{Vick1994freeGroup}, provided every member $v$ of $G$ generator $v_0$ has a unique representation of the form
\[
v = \mathop{\sum}\limits_{v_0\in G}kv_0 = \overbrace{v_0+\cdots +v_0.}^{\mbox{\textcolor{blue}{\bf $k$ moves to reach $v$ from $v_0$}}}
\]
  The Betti number for $G$ (denoted by $\beta_{\alpha}(G)$) is a count of the number of generators in $G$ (its rank)~\cite[\S 4, p. 24]{Munkres1984}.
\textcolor{blue}{\Squaresteel}
\end{definition} 

\begin{example}\label{ex:intersectingCycles} {\bf Free fg Abelian group representation for path-connected cycles}.\\
Intersecting filled cycles $\cyc A, \cyc B$ with boundaries $\bdy A, \bdy B$, interiors $\Int A,\Int B$ define shape $\sh E$, and generators $v_0,v_{0'}$, respectively, is shown in Fig.~\ref{fig:intersectingCycles}.  For each vertex $v\in \bdy(\cyc A)$, there is a minimal sequence of $k$ moves in a traversal of the edges between $v_0$ and $v$.  For this reason, we write
\[
v = \overbrace{kv_0.}^{\mbox{\textcolor{blue}{\bf $k$ moves to reach $v$ from $v_0$}}}
\]
From Def.~\ref{def:freeAbelianGroup}, we have
\begin{description}
\item [{\bf Free fg Abelian group representation}] $G\left(\langle v_0,v'_0\rangle,+\right)$. The \emph{free} property of $G$ stems from the choice of minimal path to reach a vertex $v$ from a generator $v_0$.
\item [{\bf Betti number}] $\beta_{\alpha}(\sh E) = 2$. \mbox{\textcolor{blue}{\Squaresteel}}
\end{description}
\end{example}

\begin{example}\label{ex:filledCycles}
A filled cycle $\cyc E$ with boundary $\bdy E$\ and interior $\Int E$\ containing a cyclic group generator vertex $v_0$  is shown in Fig.~\ref{fig:cycE}.  For each vertex $v\in \bdy E$, there is a minimal sequence of $k$ moves in a traversal of the edges between $v_0$ and $v$.  For this reason, we write
\[
v = \overbrace{kv_0.}^{\mbox{\textcolor{blue}{\bf $k$ moves to reach $v$ from $v_0$}}}
\]
From Def.~\ref{def:freeAbelianGroup}, we have
\begin{description}
\item [{\bf Free fg Abelian group representation}] $\cyc E\left(\langle v_0\rangle,+\right)$. 
\item [{\bf Betti number}] $\beta_{\alpha}(\cyc E) = 1$. \mbox{\textcolor{blue}{\Squaresteel}}
\end{description}
\end{example}


\begin{figure}[!ht]
\centering
\subfigure[Shape $\sh E$ boundary $\partial (\sh E)$]
 {\label{fig:shE}
\begin{pspicture}
(-1.5,-0.5)(4.0,3.0)
\psframe*[linewidth=0.75pt,linearc=0.25,cornersize=absolute,linestyle=solid,linecolor=orange!20](-1.25,-0.25)(3.25,3)
\psframe[linewidth=0.75pt,linearc=0.25,cornersize=absolute,linestyle=solid](-1.25,-0.25)(3.25,3)
\psline*[linestyle=solid,linecolor=green!30]%
(0,0)(1,0.5)(2.0,0.0)(3.0,0.5)(3.0,1.5)(2.0,2.0)(1,1.5)(0,2)
(-1,1.5)(-1,0.5)(0,0)
\psline[linestyle=solid,linecolor=black]%
(0,0)(1,0.5)(2.0,0.0)(3.0,0.5)(3.0,1.5)(2.0,2.0)(1,1.5)(0,2)
(-1,1.5)(-1,0.5)(0,0)
\psdots[dotstyle=o,dotsize=2.2pt,linewidth=1.2pt,linecolor=black,fillcolor=yellow!80]%
(0,0)(1,0.5)(2.0,0.0)(3.0,0.5)(3.0,1.5)(2.0,2.0)(1,1.5)(0,2)
(-1,1.5)(-1,0.5)(0,0)
\psline[linestyle=solid,linecolor=blue,border=1pt]{*->}(2.5,3.25)(1.0,2.50)
\rput(2.5,3.45){\footnotesize $\boldsymbol{\partial (\sh E)}$}
\rput(-1.0,2.75){\footnotesize $\boldsymbol{K}$}
\rput(0.0,1.75){\footnotesize $\boldsymbol{sh E}$}
\end{pspicture}}\hfil
\subfigure[Shape $\rb E$ boundary $\partial (\rb E)$]
 {\label{fig:rbNrvE}
\begin{pspicture}
(-1.5,-0.5)(4.0,3.0)
\psframe*[linewidth=0.75pt,linearc=0.25,cornersize=absolute,linestyle=solid,linecolor=orange!20](-1.25,-0.25)(3.25,3)
\psframe[linewidth=0.75pt,linearc=0.25,cornersize=absolute,linestyle=solid](-1.25,-0.25)(3.25,3)
\psline*[linestyle=solid,linecolor=green!30]%
(0,0)(1,0.5)(2.0,0.0)(3.0,0.5)(3.0,1.5)(2.0,2.0)(1,1.5)(0,2)
(-1,1.5)(-1,0.5)(0,0)
\psline[linestyle=solid,linecolor=black]%
(0,0)(1,0.5)(2.0,0.0)(3.0,0.5)(3.0,1.5)(2.0,2.0)(1,1.5)(0,2)
(-1,1.5)(-1,0.5)(0,0)
\psdots[dotstyle=o,dotsize=2.2pt,linewidth=1.2pt,linecolor=black,fillcolor=yellow!90]%
(0,0)(1,0.5)(2.0,0.0)(3.0,0.5)(3.0,1.5)(2.0,2.0)(1,1.5)(0,2)
(-1,1.5)(-1,0.5)(0,0)
\psline[linestyle=solid](1,1.5)(2,2)\psline[arrows=<->](-1,1.7)(-0.3,2.0)\psline[arrows=<->](0.2,2.05)(0.8,1.75)
\psline*[linestyle=solid,linecolor=orange!20]%
(0,0.25)(1,0.75)(2.0,0.25)(2.5,0.5)(2.5,0.75)(2.0,1.35)(1,1.25)(0,1.5)
(-.55,1.25)(-.55,0.75)(0,0.25)
\psline[linestyle=solid,linecolor=black]%
(0,0.25)(1,0.75)(2.0,0.25)(2.5,0.5)(2.5,0.75)(2.0,1.35)(1,1.25)(0,1.5)
(-.55,1.25)(-.55,0.75)(0,0.25)
\psdots[dotstyle=o,dotsize=2.2pt,linewidth=1.2pt,linecolor=black,fillcolor=yellow!80]%
(0,0.25)(1,0.75)(2.0,0.25)(2.5,0.5)(2.5,0.75)(2.0,1.35)(1,1.25)(0,1.5)
(-.55,1.25)(-.55,0.75)(0,0.25)
\psline[linestyle=solid,linecolor=blue,border=1pt]{*->}(2.5,3.25)(1.0,2.50)
\psline[linestyle=solid,linecolor=blue,border=1pt]{*->}(2.5,3.25)(1.0,1.00)
\rput(2.5,3.45){\footnotesize $\boldsymbol{\partial (\rb E)}$}
\rput(-1.0,2.75){\footnotesize $\boldsymbol{K}$}
\rput(0.0,1.75){\footnotesize $\boldsymbol{rb E}$}
\rput(2.8,1.85){\footnotesize $\boldsymbol{cyc A}$}
\rput(2.5,1.25){\footnotesize $\boldsymbol{cyc B}$}
\end{pspicture}}\hfil
\caption[]{Sample boundary regions $\partial (\sh E)$ and $\partial (\rb E)$}
\label{fig:boundaryRegions}
\end{figure}

\section{Shape boundary and Ribbon boundary}
Let $K$ be a CW space, bounded region $\re E$ in $K$.  Recall that a simple closed curve is a curve that has no self-loops (no self-intersections), has no endpoints and which completely encloses a surface region.  A surface region $\re E$ is defined by a cover. A collection of subsets $\mathcal{A}$ of $K$ is a cover of a bounded region $\re E$ of $K$, provided the union of the elements of $\mathcal{A}\supseteq\re E$~\cite[\S 26, p. 164]{Munkres2000}.
  
	A nonvoid region shape $\re E$ in CW space $K$ is the set of points in $K$ that are either on the bordering edge of $\re E$ (a simple closed curve denoted by $\bdy(\re E)$) or in the interior of $\re E$ (denoted by $\Int(\re E)$).  The {\bf boundary region} of $K$ exterior to $\re E$ (denoted by $\partial \sh E$) is the set of all points in $K$ not included in the closure of $\re E$.  Let $2^K$ be the collection of all subsets of space $K$.  Recall that the {\bf closure} of a nonempty set $E\in 2^K$ (denoted by $\cl E$) is the set of points $x\in K$ such that the Hausdorff distance between $x$ and $E$ (denoted by $D(x,E)$) is zero~\cite[\S 1.3,p.5]{Naimpally2013}, {\em i.e.},
\begin{align*}
\cl(\re E) &= \overbrace{\left\{x\in K: D(x,\re E) = 0\right\}.}^{\mbox{\textcolor{blue}{\bf Closure $\boldsymbol{\cl(\re E)}$ is the set all points of $\boldsymbol{x}$ in region $\boldsymbol{\re E}$.}}}\\
 &= \bdy(\re E)\cup \Int(\re E)
\end{align*}

A CW space $K$ containing a nonvoid region $\cl(\re E)$ is prescribed by
\[
K = \cl(\re E)\cup \partial(\cl(\re E)).
\]

\begin{definition}\label{def:boundary}
For the boundary $\partial \cl(\re E)$ of the closure of a region $\re E$, we have
\[
\partial \cl(\re E) = \overbrace{\left\{x\in K: x\ \not\in\ \cl(
\re E)\right\} = K\setminus \cl(\re E).}^{\mbox{\textcolor{blue}{\bf $\boldsymbol{\partial \cl(\re E)}$ is the set all points in $\boldsymbol{K}$ exterior to $\boldsymbol{\cl(\re E)}$.}}}\mbox{\textcolor{blue}{\Squaresteel}}
\]
\end{definition}

Recall that a {\bf planar shape} is a finite region of the Euclidean plane bounded by a simple closed curve (shape contour) with a nonempty interior (shape content)~\cite[p. ix]{Peters2020CGTPhysics}.

\begin{example}
Let $\sh E$ be a shape in a CW space $K$.  We identify all boundary region points that are not in the closure of shape $\sh E$ (denoted by $\partial(\cl(\sh E))$), {\em e.g.},
\[
\partial (\cl(\sh E)) = \overbrace{K\setminus \cl(\sh E).}^{\mbox{\textcolor{blue}{\bf $\boldsymbol{\partial \cl(\sh E)}$ is the set all points in $\boldsymbol{K}$ not in $\boldsymbol{\cl(\sh E)}$.}}}
\]
In other words, the boundary of a shape contains all points in $K$ that are not in the shape.  For example, in Fig.~\ref{fig:shE}, the orange region \textcolor{orange!50}{\Large \Squaresteel} represents the boundary region $\boldsymbol{\partial \cl(\sh E)}$ containing all points in $K$ that are not in the shape $\sh E$.
\textcolor{blue}{\Squaresteel}
\end{example}

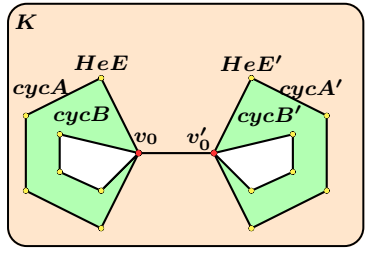
\begin{figure}[!ht]
\centering
\begin{pspicture}
(-1.5,-0.5)(4.0,3.0)

\psframe*[linewidth=0.75pt,linearc=0.25,cornersize=absolute,linestyle=solid,linecolor=orange!20](-1.25,-0.25)(3.5,3)

\psframe[linewidth=0.75pt,linearc=0.25,cornersize=absolute,linestyle=solid](-1.25,-0.25)(3.5,3)

\psline*[linestyle=solid,linecolor=green!30]%
(0,0)(0.5,1)(0,2)(-1,1.5)(-1,0.5)(0,0)(0.5,1)
\psline[linestyle=solid,linecolor=black]%
(0,0.5)(0.5,1)(0,2)(-1,1.5)(-1,0.5)(0,0)(0.5,1) 
\psline*[linestyle=solid,linecolor=white]%
(0,0.5)(0.5,1)(-.55,1.25)(-.55,0.75)(0,0.5)
\psline[linestyle=solid,linecolor=black]%
(0,0.5)(0.5,1)(-.55,1.25)(-.55,0.75)(0,0.5)
\psline[linestyle=solid,linecolor=black]%
(0.5,1)(1.5,1)
\psdots[dotstyle=o,dotsize=2.2pt,linewidth=1.2pt,linecolor=black,fillcolor=yellow!80]%
(0,0.5)(0.5,1)(-.55,1.25)(-.55,0.75)(0,0.5)
(0.5,1)(0,2)(-1,1.5)(-1,0.5)(0,0) 
\psline*[linestyle=solid,linecolor=green!30]%
(1.5,1)(2.0,0.0)(3.0,0.5)(3.0,1.5)(2.0,2.0)(1.5,1)
\psline[linestyle=solid,linecolor=black]%
(1.5,1)(2.0,0.0)(3.0,0.5)(3.0,1.5)(2.0,2.0)(1.5,1)
\psline*[linestyle=solid,linecolor=white]%
(1.5,1)(2.55,1.25)(2.55,0.75)(2.0,0.5)(1.5,1)
\psline[linestyle=solid,linecolor=black]%
(1.5,1)(2.55,1.25)(2.55,0.75)(2.0,0.5)(1.5,1)
\psdots[dotstyle=o,dotsize=2.2pt,linewidth=1.2pt,linecolor=black,fillcolor=yellow!80]%
(1.5,1)(2.0,0.0)(3.0,0.5)(3.0,1.5)(2.0,2.0)
(1.5,1)(2.55,1.25)(2.55,0.75)(2.0,0.5)(1.5,1)
\psdots[dotstyle=o,dotsize=2.5pt,linewidth=1.2pt,linecolor=black,fillcolor=red!80]
(0.5,1)(1.5,1)
\rput(-1.0,2.75){\footnotesize $\boldsymbol{K}$}
\rput(-0.0,2.2){\footnotesize $\boldsymbol{He E}$}
\rput(2.0,2.2){\footnotesize $\boldsymbol{He E'}$}
\rput(-0.8,1.85){\footnotesize $\boldsymbol{cyc A}$}
\rput(-0.25,1.5){\footnotesize $\boldsymbol{cyc B}$}
\rput(2.80,1.85){\footnotesize $\boldsymbol{cyc A'}$}
\rput(2.25,1.5){\footnotesize $\boldsymbol{cyc B'}$}
\rput(0.6,1.2){\footnotesize $\boldsymbol{v_0}$}
\rput(1.3,1.2){\footnotesize $\boldsymbol{v'_0}$}
\end{pspicture}\hfil
\caption[]{Vigolo Hawaiian earrings $\er E,\er E'$  with intersecting cycles\ \&\ attached to each other with edge $\arc{v_0,v'_0}$}
\label{fig:heE}
\end{figure}

Similarly, consider the boundary region of a planar ribbon.  Recall that a cycle (also called a 1-cycle~\cite[\S 1.12, p. 29]{Peters2020CGTPhysics}) in a CW space is a simple closed curve defined by a sequence of path-connected vertexes with no end vertex.

\begin{definition}\label{def:ribbon} {\rm Planar Ribbon~\cite{Peters2020AMSBullRibbonComplexes}}.\\
Let $\cyc A, \cyc B$ be nesting filled cycles (with $\cyc B$ in the interior of $\cyc A$) defined on a finite, bounded, planar region in a CW space $K$.  A \emph{planar ribbon} $E$ (denoted by $\rb E$) is defined by
\[
\rb E = \overbrace{
\left\{\cl(\cyc A)\setminus \left\{\cl(\cyc B)\setminus \Int(\cyc B)\right\}: \bdy(\cl(\cyc B))\subset \cl(\rb E)\right\}.}^{\mbox{\textcolor{blue}{\bf $\bdy(\cl(\cyc B))$ defines the inner boundary of $\cl(\rb E)$.}}}\mbox{\textcolor{blue}{\Squaresteel}}
\]
\end{definition}

In a CW space, each planar ribbon has a boundary region.  The boundary region of a ribbon $\rb E$ in a space $K$ contains all parts of the space not in the closure of the ribbon, {\em i.e.},
\[
\partial (\cl(\rb E)) = \overbrace{K\setminus \cl(\rb E) = \mbox{\textcolor{orange!50}{\Large \Squaresteel}}\cap \cyc A = \mbox{\textcolor{orange!50}{\Large \Squaresteel}}\cap \bdy(\cyc B) = \emptyset.}^{\mbox{\textcolor{blue}{\bf $\boldsymbol{\partial \cl(\rb E)}$ is the set all points of $\boldsymbol{K}$ not in $\boldsymbol{\cl(\rb E)}$.}}}
\]
In other words, the boundary of a ribbon contains all points in $K$ that are not in the cycles or in the interior of the ribbon.  

\begin{example}
In Fig.~\ref{fig:rbNrvE}, the two orange regions (\textcolor{orange!50}{\Large \Squaresteel} surrounding the ribbon border cycle $\cyc A$ and \textcolor{orange!50}{\Large \Squaresteel} inside the region of $K$ in the interior of the inner cycle $\cyc B$, which is not part of the ribbon) represent the boundary region $\boldsymbol{\partial \cl(\rb E)}$ containing all points in $K$ that are not in the ribbon $\rb E$.
\textcolor{blue}{\Squaresteel}
\end{example}

A pair of Vigolo Hawaiian earrings\footnote{introduced by F. Vigolo~\cite[\S 3.31, p. 79]{Vigolo2018HawaiianEarrings}.} have inner and outer borders constructed from a pair intersecting symmetric smooth curves.  A {\bf Vigolo Hawaiian earring} is a modified Brooks wide ribbon constructed from the intersection of a pair of smooth curves.  A {\bf Brooks wide ribbon}~\cite{Brooks2013wideRibbons} is constructed from a pair of smooth, non-intersecting, closed curves which are boundaries wrapped around a spatial region. Unlike a Brooks wide ribbon, a Vigolo Hawaiian earring is a wide ribbon with inner and outer border curves that intersect.  To carry this a step further, a planar wide ribbon in a CW space has borders that are possibly intersecting 1-cycles.  In a CW space, a wide ribbon has an outer ribbon border that is a partially filled 1-cycle and an inner border that is a non-filled 1-cycle.  Unlike either a Brooks wide ribbon~\cite{Brooks2013wideRibbons} or a Vigolo Hawaiian earring~\cite{Vigolo2018HawaiianEarrings}
, the interior of CW space wide ribbon is nonempty.  By definition, a CW space wide ribbon complex is an example of a shape, {\em i.e.}, a {\bf ribbon shape} is a region of the space with an outer border that is a simple closed curve with a non-empty interior (our filled 1-cycle).  

\begin{example}
The orange region \textcolor{orange!50}{\Large \Squaresteel} of the CW space $K$ external to the outer borders $\cyc A, \cyc A'$ of the Hawaiian earrings $\er E,\er E'$ as well the two symmetric white regions  enclosed by 
inner earring border cycles $\cyc B, \cyc B'$ in Fig.~\ref{fig:heE} constitute the complete border region of the earrings, {\em i.e.}\\
\[
\partial\left(\er E\cup\er E'\right) = K\setminus (\er E\cup\er E').
\]
\textcolor{blue}{\Squaresteel}
\end{example}

\begin{theorem}\label{thm:amenableRibbons}{\rm \cite{PetersVergili2020descriptiveFixedSets}}
Let $(K,\delta_{\Phi})$ be a descriptive proximity space over a CW space $K$ with probe function $\Phi: 2^K \to \mathbb{R}^n$, ribbon $\rb A\in 2^K$, and  $f : (K,\delta_{\Phi})\to (K,\delta_{\Phi})$ a descriptive proximally continuous map such that $f(\rb A) \ \dcap \ \rb A \neq \emptyset$.  Then $\rb A, f(\rb A)$ are amiable fixed sets.
\end{theorem}
\begin{proof}
Immediate from Def.~\ref{def:desProxCont}.
\end{proof}

\begin{example}
A pair of Hawaiian earrings $\er E,\er E'$ in a CW space $K$ are represented in Fig.~\ref{fig:heE}.  For a nonvoid region $\re E = \er E\cup\er E'\in K$, let $\langle v_0,v'_0\rangle$ be generator vertexes in a free Abelian group representation of the earrings and let $\beta_{\re E}$ be the Betti number~\cite{Munkres1984} defined by
\[
\beta_{\re E} =\overbrace{\mbox{rank of group 
$G_{\re E}(\left\{\langle v_0,v'_0\rangle\right\},+)$.}}^{\mbox{\textcolor{blue}{\bf
No. of generators in free fg Abelian group representing $\er E\cup \er E'$.}}}
\]
Also, let $f : (K,\delta_{\Phi})\to (K,\delta_{\Phi})$ be a descriptive proximally continuous map defined by
\[
f(\re E) = K\setminus \partial (\re E).
\]
Then $f(\er E\cup \er E') = K\setminus\partial(\er E\cup \er E') = \er E\cup \er E'$, {\em i.e.}, $\er E\cup \er E'$ is a fixed subset of $f$.  We also have
$f(\er E\cup \er E')\ \dcap\ K\setminus\partial(\er E\cup \er E')\neq \emptyset$, since
\[
\Phi(\er E\cup \er E') = \Phi(K\setminus\partial(\er E\cup \er E')) = 2.
\]
Hence, from Theorem~\ref{thm:amenableRibbons}, $f(\er E\cup \er E'), \er E\cup \er E'$ are amiable fixed sets.
\textcolor{blue}{\Squaresteel}
\end{example}

\section{Main results} 

Our observations about the boundary region of a cell complex in a CW space lead to a variation of the Jordan curve theorem (see Theorem~\ref{thm:JordanCurveTheorem}), which launches P. Alexandroff's introduction to Algebraic Topology~\cite{Alexandroff1932elementaryConcepts}.

\begin{theorem}\label{thm:JordanCurveTheorem} {\rm [Jordan Curve Theorem~\cite{Jordan1893coursAnalyse}]}.\\
A simple closed curve lying on the plane divides the  
plane into two regions and forms their common boundary.
\end{theorem}
\begin{proof}
For the first complete proof, see O. Veblen~\cite{Veblen1905TAMStheoryOFPlaneCurves}.  For a simplified proof via the Brouwer Fixed Point Theorem, see R. Maehara~\cite{Maehara1984AMMJordanCurvedTheoremViaBrouwerFixedPointTheorem}.  For an elaborate proof, see J.R. Munkres~\cite[\S 63, 390-391, Theorem 63.4]{Munkres2000}.
\end{proof}

The boundary region of any planar shape contains all cell complexes that are not part of the shape in a CW space. 

\begin{theorem}\label{thm:cellComplex}{\rm \bf Shape Boundary Region Jordan Curve Theorem}
The boundary of a nonempty planar shape $\sh E$ in a CW space $K$ separates the space into two disjoint regions.
\end{theorem}
\begin{proof}
From Def.~\ref{def:boundary}, there are two disjoint regions of a CW space $K$, namely, $\partial \cl(\re E)$ and $\cl(\re E)$.  A shape $\sh E$ in $K$ is a nonvoid region of $K$.  From Theorem~\ref{thm:filledCycle}, $\sh E$ is a filled cycle defined by its closure, namely, the union of the shape perimeter, which is a simple, closed curve and its nonempty shape interior. That is, by definition, we have
\begin{align*} 
K &= \cl(\sh E) \cup \partial (\cl(\sh E)),\ \mbox{and}\\
\partial (\cl(\sh E)) &= K\setminus \cl(\sh E).
\end{align*}
Hence, $\cl(\sh E) \cap \partial \cl(\sh E) = \emptyset$, {\em i.e.}, $\cl(\sh E)$ and $\partial \cl(\sh E)$ are disjoint regions whose union is $K$.
\end{proof}

\begin{figure}[!ht]
\centering
\begin{pspicture}
(-1.8,-0.5)(3.8,3.0)
\psframe[linewidth=0.75pt,linearc=0.25,cornersize=absolute,linecolor=blue](-1.55,-0.25)(9,3.0)
\psline*[linestyle=solid,linecolor=green!30]%
(0,0)(1,0.5)(2.0,0.0)(3.0,0.5)(3.0,1.5)(2.0,2.0)(1,1.5)(0,2)
(-1,1.5)(-1,0.5)(0,0)
\psline[linestyle=solid,linecolor=black]%
(0,0)(1,0.5)(2.0,0.0)(3.0,0.5)(3.0,1.5)(2.0,2.0)(1,1.5)(0,2)
(-1,1.5)(-1,0.5)(0,0)
\psdots[dotstyle=o,dotsize=2.2pt,linewidth=1.2pt,linecolor=black,fillcolor=yellow!80]%
(0,0)(1,0.5)(2.0,0.0)(3.0,0.5)(3.0,1.5)(2.0,2.0)(1,1.5)(0,2)
(-1,1.5)(-1,0.5)(0,0)
\psline*[linestyle=solid,linecolor=white]%
(0,0.25)(1,0.5)(2,0.25)(3.0,1.5)(2.5,1.25)(2.0,1.5)(1,1.25)(0,1.5)
(-.55,1.25)(-1,1.5)(0,0.25)
\psline[linestyle=solid,linecolor=black]%
(0,0.25)(1,0.5)(2,0.25)(3.0,1.5)(2.5,1.25)(2.0,1.5)(1,1.25)(0,1.5)
(-.55,1.25)(-1,1.5)(0,0.25) 
\psdots[dotstyle=o,dotsize=2.2pt,linewidth=1.2pt,linecolor=black,fillcolor=yellow!80]%
(0,0.25)(1,0.5)(2,0.25)(2.0,0.25)(3.0,1.5)(1,1.25)(1,1.25)(0,1.5)
(-.55,1.25)(-1,1.5)(0,0.25)(2.0,1.5)(2.5,1.25)
\psdots[dotstyle=o,dotsize=2.5pt,linewidth=1.2pt,linecolor=black,fillcolor=red!80]
(-1,1.5)(3.0,1.5)(1,0.5)
\rput(-1.0,2.75){\footnotesize $\boldsymbol{K}$}
\rput(-1.15,1.6){\footnotesize $\boldsymbol{g}$}
\rput(3.2,1.6){\footnotesize $\boldsymbol{g'}$}
\rput(1,0.20){\footnotesize $\boldsymbol{v'_0}$}
\rput(1.0,1.8){\footnotesize $\boldsymbol{Hn E}$}
\rput(2.8,1.85){\footnotesize $\boldsymbol{cyc A}$}
\rput(1.8,1.25){\footnotesize $\boldsymbol{cyc B}$}
\end{pspicture}
\begin{pspicture}
(-1.5,-0.5)(4.0,4.0)
\psline*[linestyle=solid,linecolor=green!30]%
(0,0)(1,1)(0,2)(-1,1.5)(-1,0.5)(0,0)
\psline[linestyle=solid,linecolor=black]%
(0,0.0)(1,1)(0,2)(-1,1.5)(-1,0.5)(0,0)
\psline*[linestyle=solid,linecolor=orange!50]%
(0,0.5)(1,1)(-.55,1.25)(-.55,0.75)(0,0.5)
\psline[linestyle=solid,linecolor=black]%
(0,0.5)(1,1)(-.55,1.25)(-.55,0.75)(0,0.5)
\psdots[dotstyle=o,dotsize=2.2pt,linewidth=1.2pt,linecolor=black,fillcolor=yellow!80]%
(0,0.5)(1,1)(-.55,1.25)(-.55,0.75)(0,0.5)
(0,2)(-1,1.5)(-1,0.5)(0,0) 
\psline*[linestyle=solid,linecolor=green!30]%
(1,1)(2.0,2.0)(3.0,1.5)(3.0,0.5)(2.0,0.0)(1,1)
\psline[linestyle=solid,linecolor=black]%
(1,1)(2.0,0.0)(3.0,0.5)(3.0,1.5)(2.0,2.0)(1,1)
\psline*[linestyle=solid,linecolor=orange!50]%
(1,1)(2.55,1.25)(2.55,0.75)(2.0,0.5)(1,1)
\psline[linestyle=solid,linecolor=black]%
(1,1)(2.55,1.25)(2.55,0.75)(2.0,0.5)(1,1)
\psdots[dotstyle=o,dotsize=2.2pt,linewidth=1.2pt,linecolor=black,fillcolor=yellow!80]%
(1,1)(2.0,0.0)(3.0,0.5)(3.0,1.5)(2.0,2.0)
(2.55,1.25)(2.55,0.75)(2.0,0.5)
\psline*[linestyle=solid,linecolor=black!50]%
(1,1)(0.9,0.8)(0.9,0.5)(1.0,0.3)(1.1,0.5)(1.1,0.8)(1,1)
\psdots[dotstyle=o,dotsize=2.2pt,linewidth=1.2pt,linecolor=black,fillcolor=yellow!80]%
(1,1)(0.9,0.8)(0.9,0.5)(1.0,0.3)(1.1,0.5)(1.1,0.8)(1,1)
\psline[linestyle=solid,linecolor=blue,border=1pt]{*-}(0.8,1.5)(1,1)
\psline[linestyle=solid,linecolor=blue,border=1pt]{*-}(1.2,1.5)(1,1)
\psdots[dotstyle=o,dotsize=3.5pt,linewidth=1.2pt,linecolor=black,fillcolor=red!80]
(0.8,1.5)(1.2,1.5)(1.0,1)
\rput(1.0,2.1){\footnotesize $\boldsymbol{\Hb E}$}
\rput(-0.8,1.85){\footnotesize $\boldsymbol{cyc Ha}$}
\rput(-0.25,1.5){\footnotesize $\boldsymbol{cyc Hb}$}
\rput(2.85,1.85){\footnotesize $\boldsymbol{cyc Ha'}$}
\rput(2.25,1.5){\footnotesize $\boldsymbol{cyc Hb'}$}
\rput(1.35,0.95){\footnotesize $\boldsymbol{v_0}$}
\rput(0.7,1.7){\footnotesize $\boldsymbol{v_1}$}
\rput(1.3,1.7){\footnotesize $\boldsymbol{v_2}$}
\end{pspicture}
\caption[]{Hawaiian necklace $\boldsymbol{\Hn E}$\ \&\ Hawaiian butterfly $\boldsymbol{\Hb E}$}
\label{fig:HawaiianWideRibbons}
\end{figure}

\begin{definition}{\rm \bf Descriptive Closure}.\\
Let $2^K$ be a collection of cell complexes, $(K,\dnear)$ a descriptive proximity space, $E\in 2^K$, $\Phi(E)$ a description of $E$.  The descriptive closure of $E$ (denoted by $\cl_{\Phi}(E))$ is defined by
\[
\dcl(E) = \left\{A\in 2^K: A\ \dnear\ E\right\}.\mbox{\textcolor{blue}{\Squaresteel}}
\]
\end{definition}

\begin{theorem}\label{thm:Brouwer3}{Fixed Cell Complex Theorem~}$\mbox{}$\\
Let $K$ be a planar CW space.
Every descriptive proximally continuous map from $K$ to itself has a fixed cell complex.
\end{theorem} 
\begin{proof}
Let $\left(K,\dnear\right)$ be a descriptive proximity space. Also, let $\dcl(\sh E)\in \Phi(K)$ be the descriptive closure of a shape $\sh E$ and let $\Phi(\partial(\sh E))\in \Phi(K)$ be the description of the boundary region in $\dcl(\sh E)$.  In addition, let $f:\left(K,\dnear\right) \to \left(K,\dnear\right)$ be a descriptive proximally continuous map, defined by
\begin{align*}
\Phi(K) &= \left\{\Phi(A): A\in 2^K\right\}.\\
\Phi(\partial(\cl(\sh E))) &= \left\{\Phi(A): A\in 2^{\partial(\cl(\sh E))}\right\}.\\
f(\dcl(\sh E)) &= \Phi(K)\setminus \Phi(\partial(\cl(\sh E)))\ \mbox{such that} \\ 
\Phi(\partial(\cl(\sh E))) &= \Phi(K)\setminus \dcl(\sh E)\ \mbox{(from Def.~\ref{def:boundary}\ \&\ Theorem~\ref{thm:cellComplex})},\ \mbox{\em i.e.,}
\end{align*}
$\Phi(\partial(\cl(\sh E)))$ contains the descriptions of all cell complexes not in the descriptive closure of shape $\sh E$ and $\Phi(K)\setminus \dcl(\sh E)$ contains that part of $\Phi(K)$ not in  $\Phi(\partial(\cl(\sh E)))$, namely, $\dcl(\sh E)$.  Hence, $f(\dcl(\sh E)) = \dcl(\sh E)$.
\end{proof}


\begin{theorem}\label{thm:Brouwer5}{Planar Wide Ribbon Fixed Set Theorem~}$\mbox{}$\\
Every descriptive proximally continuous map from a planar wide ribbon to itself in a CW space has a fixed set.
\end{theorem} 
\begin{proof}
Let $\rb E$ be a wide ribbon in a planar CW space $K$.  Replace $K$ with $\rb E$ in the proof of Theorem~\ref{thm:Brouwer3} and the desired result follows.
\end{proof}

\begin{example}
From Theorem~\ref{thm:Brouwer5}, the Hawaiian necklace $\Hn E$ and Hawaiian butterfly $\Hb E$ wide ribbons in Fig.~\ref{fig:HawaiianWideRibbons} are fixed subsets of a descriptive proximally continuous function $f$. 
	\textcolor{blue}{\Squaresteel} 
\end{example}

For a CW space $K$ containing a cell complex $\sh E$, let $\left(K,\dnear\right)$ be a descriptive proximity space and let $f:\left(K,\dnear\right)\to \left(K,\dnear\right)$ be a descriptive proximally continuous map.  In addition, we have
\begin{align*}
\Phi(K) &= \left\{\Phi(A): A\in K\right\}.\\
\Phi(K\setminus \sh E) &= \left\{\Phi(B): B\in K\ \&\ B\not\in \sh E\right\}.
\end{align*}

\begin{theorem}
let $f:\left(K,\dnear\right)\to \left(K,\dnear\right)$ be a descriptive proximally continuous map for a CW space $K$ containing a cell complex $\sh E$, defined by $f(\cl(\sh E)) = K\setminus \partial\cl(\sh E)$.  Then 
\begin{compactenum}[1$^o$]
\item A shape boundary $\partial(\cl(\sh E))$ is a fixed set subset of $f$.
\item $f(\partial(\cl(\sh E))), \partial(\cl(\sh E))$ are amiable.
\end{compactenum}
\end{theorem}
\begin{proof}
1$^o$: By Def. 2, $\partial(\cl(\sh E)))\in K$. Then
\[
f(\partial(\cl(\sh E))) = K\setminus \partial(\partial\cl(\sh E)) = \partial(\cl(\sh E)).
\]
Hence, the boundary $\partial(\cl(\sh E))$ is a fixed set subset of $f$.\\
2$^o$: 
From 1$^o$, $f(\partial(\cl(\sh E))) = \partial(\cl(\sh E))$ implies 
\[
\Phi(f(\partial(\cl(\sh E)))) = \Phi(\partial(\cl(\sh E))).
\]
Then, $f(\partial(\cl(\sh E)))\ \dcap\ \partial(\cl(\sh E))\neq \emptyset$.
Hence, $f(\partial(\cl(\sh E))), \partial(\cl(\sh E))$ are amiable.
\end{proof}

In considering applications of amiable shapes, we relax the amiability matching description requirement. In practice, physical shapes seldom have descriptions that exactly match.  Hence, we consider cases where shapes with descriptions are, in some sense, close.  For example, Balko-P\'{o}r-Scheucher-Swanepoel-Valtr almost equidistant sets~\cite{Balko2017almostEquidistantSets} have a counterpart in what are known as almost amiable planar shapes.  In a 2-dimensional Euclidean space, a set of points are {\bf almost equidistant}, provided, for any three points from the set, a least two of the points are distance 1 apart.  For example, the pair of Hawaiian earrings in Fig.~\ref{fig:heE} are a distance of 1 edge apart.

\begin{definition}\label{def:almostAmiable}{\rm {\bf Almost Amiable Shapes.}}\\
	Let $(X,\delta_{\Phi})$ be a descriptive  \v{C}ech proximity space and $f: (X, \delta_{\Phi}) \to (X, \delta_{\Phi})$ a descriptive proximally continuous map. Planar shapes $\sh E, \sh E'\in X$ are {\bf almost amiable}, provided $\Phi(f(\sh E)),\Phi(f(\sh E')$ are real numbers and
\[
\abs{\Phi(f(\sh E)) - \Phi(f(\sh E'))} \leq th.
\]
\noindent for threshold $th>0$.
	\textcolor{blue}{\Squaresteel}
\end{definition}

A more general view of almost amiable shapes (not considered, here) would be to consider descriptions of shapes that are $n$-dimensional vectors in $\mathbb{R}^n$, instead of single-valued (real) descriptions.

\begin{theorem}\label{thm:Brouwer8}{\rm (Almost Amiable Cell Complexes with Close Betti Numbers)}$\mbox{}$\\
Let $f:\left(K,\dnear\right)\to \left(K,\dnear\right)$ be a descriptive proximally continuous map for a CW space $K$ containing cell complexes with free fg Abelian group representations and with descriptions that are close Betti numbers.  Then the cell complexes are almost amiable.
\end{theorem} 
\begin{proof}
Let $\sh E,\sh E'\in K$ be cell complexes with free fg Abelian group representations and with descriptions that are close Betti numbers $\beta_{\alpha}(f(\sh E)),\beta_{\alpha}(f(\sh E'))$ in CW space $K$ and let threshold $th>1$.  Assume
\begin{align*}
f(\sh E) &= \sh E \neq f(\sh E') = \sh E',\ \mbox{and}\\
\Phi(f(\sh E)) &= \beta_{\alpha}(f(\sh E))\  \mbox{and}\\
\Phi(f(\sh E')) &= \beta_{\alpha}(f(\sh E'))\  \mbox{such that}
\end{align*}
\[
\abs{\Phi(f(\sh E)) - \Phi(f(\sh E'))} \leq th.
\]
Then, from Def.~\ref{def:almostAmiable}, $\sh E,\sh E'$ are almost amiable.
\end{proof}

Geometrically, shapes with close Betti numbers describe almost amiable structures, provided the shapes have almost the same number of generators in their free fg Abelian group representations.  This is a common occurrence among members of the same class of shapes such as the wide ribbon models of different types of Hawaiian jewelry.

\begin{example}\label{ex:HawaiianComplexes}
From Example~\ref{ex:HawaiianComplexes} the Hawaiian necklace $\Hn E$ and Hawaiian butterfly $\Hb E$ wide ribbons in Fig.~\ref{fig:HawaiianWideRibbons} each is a fixed subset of a descriptive proximally continuous function $f$.  In addition, $\Hn E,\Hb E$ with generator vertexes $g,g',v'_0$ and $v_0,v_1,v_2$ have free fg Abelian group representations $G(\langle g,g',v'_0\rangle,+),G(\langle v_0,v_1,v_2\rangle,+)$.  Let the description of each of these ribbons be equal to the Betti number that corresponds to their group representations.  This gives us
\begin{align*}
f(\Hn E) &= \Hn E = f(\Hb E) = \Hb E.\  \mbox{Then},\\
\Phi(f(\Hn E)) &= \beta_{\alpha}(f(\Hn E))\  \mbox{and}\\
\Phi(f(\Hb E)) &= \beta_{\alpha}(f(\Hb E))\  \mbox{such that}\\
\abs{\beta_{\alpha}(f(\Hn E)) - \beta_{\alpha}(f(\Hb E))} &= 0.
\end{align*}
Hence, from Theorem~\ref{thm:Brouwer8}, $\Hn E$ and $\Hb E$ are both amiable and almost amiable for any threshold $th>0$.

The Hawaiian earrings $\er E$ in Fig.~\ref{fig:heE} and the Hawaiian necklace $\Hn E$ in 
Fig.
~\ref{fig:HawaiianWideRibbons} are almost amiable but not amiable for threshold $th \geq 1$.  For example, we have
\begin{align*}
f(\er E) &= \er E \neq f(\Hb E) = \Hb E.\  \mbox{Also,},\\
\Phi(f(\er E)) &= \beta_{\alpha}(f(\er E)) = 2\  \mbox{and}\\
\Phi(f(\Hn E)) &= \beta_{\alpha}(f(\Hn E)) = 3\  \mbox{such that}\\
\abs{\Phi(f(\er E)) - \Phi(f(\Hn E))} & = 1.
\end{align*}
Hence, $\er E$ and $\Hn E$ are almost amiable.  However, $f(\er E),\Hn E$ are not amiable, since $f(\er E)\ \dcap\ \Hn E = \emptyset$.
\textcolor{blue}{\Squaresteel} 
\end{example}

\section*{Acknowledgements}
\noindent Many thanks to Tane Vergili for her many very helpful suggestions for this paper. 

\bibliographystyle{amsplain}
\bibliography{NSrefs}

\end{document}